\documentclass[a4paper,11pt]{article}

\usepackage[dvips]{graphics, epsfig, color}
\usepackage{graphics}
\usepackage{graphicx}
 \usepackage{amsfonts}

\renewcommand{\title}[1]{\vspace{\fill}
\eject\addtolength{\baselineskip}{4pt}
{\bfseries\LARGE #1}\\[3mm]\addtolength{\baselineskip}{-4pt}}
\renewcommand{\author}[3]{\parbox[t]{75mm}
{\begin{center}{\scshape #1}\\[3mm] #2\\
 {\ttfamily #3} \end{center}}}

\newtheorem{thm}{\bfseries Theorem}
\newtheorem{remark}[thm]{\bfseries Remark}    
\newtheorem{prop}[thm]{\bfseries Proposition} 

\newtheorem{conj}[thm]{\bfseries Conjecture}

\newenvironment{proof}{\medskip                    
\noindent{\scshape Proof:}}{\quad $\Box$\medskip}  

\begin{document}

\thispagestyle{empty}

\begin{center}

\title{Hyperplane Arrangements\\ with Large Average Diameter}

\author{
Antoine Deza $\quad$ and $\quad$ Feng Xie}{September 23, 2007 \\$\:$\\
McMaster University\\
Hamilton, Ontario, Canada }{ deza, xief @mcmaster.ca }
\end{center}

\begin{quote}
{\bfseries Abstract:}
The largest possible average diameter of a bounded cell of a simple
hyperplane arrangement is conjectured to be not greater than the
dimension. We prove that this conjecture holds in dimension 2,
and is asymptotically tight in fixed dimension.
We give the exact value of the largest possible average diameter for
all simple arrangements in dimension 2, for arrangements having at
most the dimension plus 2 hyperplanes, and for arrangements having
6 hyperplanes in dimension 3. In dimension 3, we
give lower and upper bounds which are both asymptotically
equal to the dimension.
\end{quote}

\begin{quote}
\end{quote}
\vspace{5mm}

\section{Introduction}
Let $\mathcal{A}$ be a simple arrangement formed by
$n$ hyperplanes in dimension $d$. We recall that an arrangement is
called simple if $n\geq d+1$ and any  $d$ hyperplanes intersect at a unique
distinct point. The number of bounded cells (closures of the bounded
connected components of the complement) of
$\mathcal{A}$ is $I={n-1\choose d}$. Let
$\delta(\mathcal{A})$ denote the average diameter of a bounded cell $P_i$ of
$\mathcal{A}$; that is,
$$\delta(\mathcal{A})=\frac{\sum_{i=1}^{i=I}\delta(P_i)}{I}$$
where  $\delta(P_i)$ denotes the diameter of $P_i$, i.e.,  the
smallest number such that any two vertices of  $P_i$ can be
connected by a path with at most $\delta(P_i)$ edges. Let
$\Delta_{\mathcal{A}}({d,n})$ denote the largest possible average
diameter of a bounded cell of a simple arrangement defined by $n$
inequalities in dimension $d$. Deza, Terlaky and
Zinchenko conjectured  that
$\Delta_{\mathcal{A}}({d,n})\leq d$.
\begin{conj}{\em \cite{DTZ07}}\label{dShub}
The average diameter of a bounded cell of a simple arrangement defined by $m$
inequalities in dimension $n$ is not greater than $n$.
\end{conj}
It was showed in~\cite{DTZ07} that if the
conjecture of Hirsch holds for polytopes in dimension $d$, then
$\Delta_{\mathcal{A}}(d,n)$  would satisfy
$\Delta_{\mathcal{A}}(d,n)\leq d+\frac{2d}{n-1}$. In dimension $2$
and $3$, we have $\Delta_{\mathcal{A}}(2,n)\leq
2+\frac{2}{n-1}$ and $\Delta_{\mathcal{A}}(3,n)\leq
3+\frac{4}{n-1}$. We recall that a polytope is a bounded polyhedron
and that the conjecture of Hirsch, formulated in 1957 and reported
in~\cite{D63}, states that the diameter of a polyhedron defined by
$n$ inequalities in dimension $d$ is not greater than $n-d$. The
conjecture does not hold for unbounded polyhedra.

Conjecture~\ref{dShub} can be regarded a discrete analogue of a result of
Dedieu, Malajovich and Shub~\cite{DMS05} on the average total curvature
of the central path associated to a bounded cell of a simple arrangement.
We first recall the definitions of the
central path and of the total curvature. For a polytope $P=\{x:
Ax\geq b\}$ with $A\in\Re^{n\times d}$, the central path
corresponding to $\min\{c^Tx:x\in P\}$ is a set of minimizers of
$\min\{c^Tx+\mu f(x):x\in P\}$ for $\mu\in(0,\infty)$ where
$f(x)=-\sum_{i=1}^n\ln (A_ix-b_i)$ -- the standard logarithmic
barrier function~\cite{RTV97}.
Intuitively, the total curvature~\cite{S90} is a measure of how far
off a certain curve is from being a straight line. Let
$\psi:[\alpha,\beta]\rightarrow\Re^d$ be a
$C^2((\alpha-\varepsilon,\beta+\varepsilon))$ map for some
$\varepsilon>0$ with a non-zero derivative in $[\alpha,\beta]$.
Denote its arc length by $l(t)=\int_{\alpha}^t
\|\dot{\psi}(\tau)\|d\tau$, its parametrization by the arc length by
$\psi_{\rm arc}=\psi\circ l^{-1}: [0,l(\beta)]\rightarrow\Re^d$, and
its curvature at the point $t$ by $\kappa(t)=\ddot{\psi}_{\rm
arc}(t)$. The total curvature is defined as $\int_0^{l(\beta)}
\|\kappa(t)\|dt$. The requirement $\dot{\psi}\neq0$ insures that any
given segment of the curve is traversed only once and allows to
define a curvature at any point on the curve. Let $\lambda^c({\mathcal A})$
denote the average associated total curvature
of a  bounded cell $P_i$ of a simple arrangement ${\mathcal A}$; that is,
$$\lambda^c({\mathcal A})=\sum_{i=1}^{i=I}\frac{\lambda^c(P_i)}{I}$$
where $\lambda^c(P)$ denotes the total curvature of the central path
corresponding to the linear optimization problem $\min \{ c^Tx :
x\in P\}$.  Dedieu,
Malajovich and Shub~\cite{DMS05} demonstrated that
$\lambda^c({\mathcal A})\leq 2\pi d$ for any fixed $c$.
Keeping the linear optimization approach but replacing central path
following interior point methods by simplex methods, Haimovich's probabilistic
analysis of the shadow-vertex simplex algorithm, see~\cite[Section 0.7]{B87},
showed that the expected number of pivots is bounded by $d$.
Note that while Dedieu, Malajovich and Shub consider only the bounded
cells (the central path may not be defined over some unbounded ones),
Haimovich considers the average over bounded and unbounded cells.
While the result of Haimovich and Conjecture~\ref{dShub} are similar
in nature, they differ in some aspects:
Conjecture~\ref{dShub}  considers the average over bounded cells, and
the number of pivots could be smaller than the diameter for some cells.

In Section~\ref{dd} we consider a simple hyperplane arrangement ${\mathcal A}^*_{\:d,n}$
combinatorially equivalent to the cyclic hyperplane arrangement which is dual to
the cyclic polytope, see~\cite{FR01} for some combinatorial properties of the (projective) cyclic hyperplane arrangement.
We show that the bounded cells of ${\mathcal A}^*_{\:d,n}$ are mainly combinatorial cubes
and, therefore, that the dimension $d$ is
an asymptotic lower bound for $\Delta_{\cal A}(d,n)$ for fixed $d$.
In Section~\ref{2d}, we consider the arrangement ${\mathcal A}^o_{\:2,n}$
resulting from the addition of one hyperplane to ${\mathcal A}^*_{\:2,n-1}$
such that all the vertices are on one side of the added hyperplane.
We show that the arrangement ${\mathcal A}^o_{\:2,n}$ maximizes the average
diameter and, thus, Conjecture~\ref{dShub} holds in dimension 2.
In Section~\ref{3d}, considering a $3$-dimensional analogue, we give lower and
upper bounds asymptotically equal to 3 for $\Delta_{{\mathcal A}}(3,n)$.
The combinatorics of the addition of a (pseudo) hyperplane to the cyclic hyperplane arrangement
is studied in details in~\cite{Z93}. For example,  the arrangements ${\mathcal A}^*_{\:2,6}$
and ${\mathcal A}^o_{\:2,6}$ correspond to the top and bottom elements of the
higher Bruhat order $B(5,2)$ given in Figure 3 of~\cite{Z93}.
For polytopes and arrangements, we refer to the books of Edelsbrunner~\cite{E87},
Gr\"unbaum~\cite{G03} and Ziegler~\cite{Z95}. 

\section{Line Arrangements with Maximal Average Diameter} \label{2d}
For $n\geq 4$, we consider the simple line arrangement
$\mathcal{A}^o_{\:2,n}$ made of the $2$ lines $h_1$ and $h_2$
forming, respectively, the $x_1$ and $x_2$ axis, and the $(n-2)$ lines
defined by their intersections with $h_1$ and $h_2$. We have
$h_k\cap h_1=\{1+(k-3)\varepsilon,0\}$ and $h_k\cap
h_2=\{0,1-(k-3)\varepsilon\}$ for $k=3,4,\dots,n-1$, and $h_n\cap
h_1=\{2,0\}$ and $h_n\cap h_1=\{0,2+\varepsilon\}$ where
$\varepsilon$ is a constant satisfying
$0<\varepsilon<1/(n-3)$. See Figure~\ref{A072} for an
arrangement combinatorially equivalent to  $\mathcal{A}^o_{\:2,7}$.

\begin{figure}[hbt]
\begin{center}
\epsfig{file=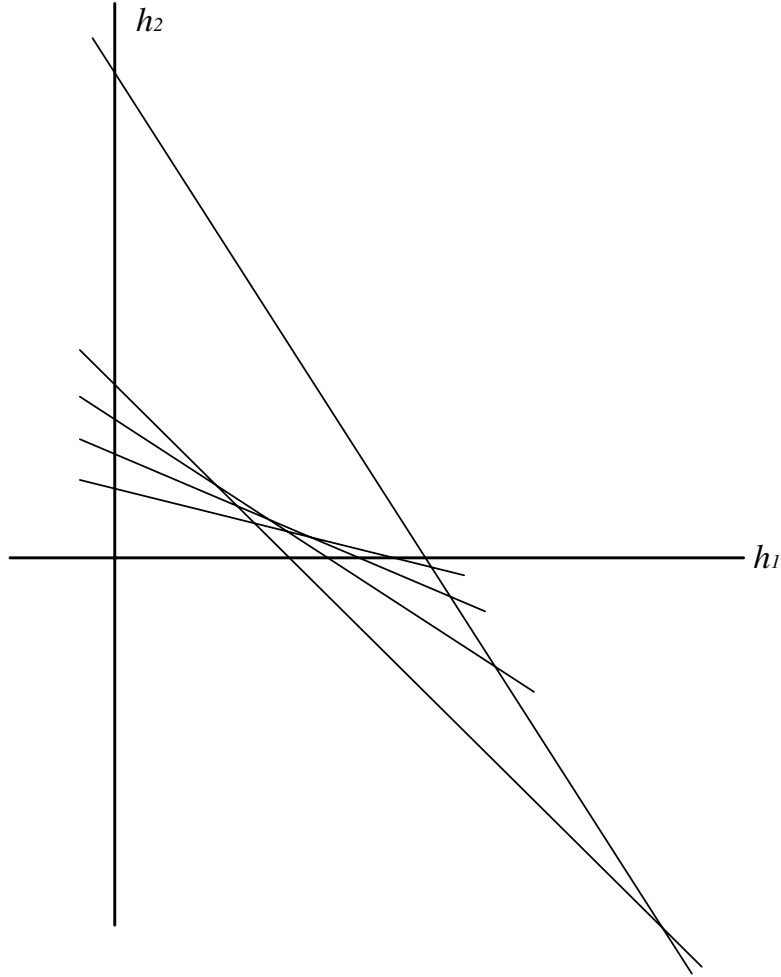, height=13.0cm} \caption{An arrangement
combinatorially equivalent to $\mathcal{A}^o_{\:2,7}$} \label{A072}
\end{center}
\end{figure}

\begin{prop}
For $n\geq 4$, the bounded cells of the  arrangement ${\mathcal
A}^{o}_{\:2,n}$ consist of $(n-2)$ triangles, $\frac{(n-1)(n-4)}{2}$
4-gons, and $1$ $n$-gon. We have $\delta({\mathcal
A}^{o}_{\:2,n})=2-\frac{2\lceil\frac{n}{2}\rceil}{(n-1)(n-2)}$ for
$n\geq 4$.
\end{prop}
\begin{proof}
The first $(n-1)$ lines of ${\mathcal A}^{o}_{\:2,n}$ clearly form a
simple line arrangement $\mathcal{A}^*_{\:2,n-1}$
which bounded cells are $(n-3)$ triangles
and ${n-3 \choose 2}$ 4-gons. The last line $h_n$ adds 1 $n$-gon,
1 triangle and $(n-4)$ 4-gons.
Since the diameter of a $k$-gon is $\lfloor\frac{k}{2}\rfloor$, we
have $\delta({\mathcal
A}^{o}_{\:2,n})=2-2\frac{(n-2)-(\lfloor\frac{n}{2}\rfloor-2)}{(n-1)(n-2)}=2-\frac{2\lceil\frac{n}{2}\rceil}{(n-1)(n-2)}$.
\end{proof}

Exploiting the fact that a line arrangement contains at least $n-2$ triangles (at least $n-d$ simplices
for a simple hyperplane arrangement~\cite{S79}) and a bound on the number
of facets on the boundary of the union of the bounded cells,  we can show that ${\mathcal A}^{o}_{\:2,n}$
attains the largest possible average diameter of a simple line arrangement.
\begin{prop}
For $n\geq 4$, the largest possible average diameter of a bounded cell of a simple
line arrangement satisfies  $\Delta_{{\mathcal A}}(2,n)=2-\frac{2\lceil\frac{n}{2}\rceil}{(n-1)(n-2)}$.
\end{prop}
\begin{proof}
Let $f_1({\mathcal A})$ denote the number of bounded edges of a simple arrangement ${\mathcal A}$
of $n$ lines, and
let $f_1(P_i)$ denote the number of edges of a bounded cell $P_i$ of ${\mathcal A}$.
Let call an edge of ${\mathcal A}$ {\em external}  if it belongs to exactly one bounded cell,
and let $f_1^0({\mathcal A})$ denote the number of external edges of ${\mathcal A}$.
Let $p_{odd}({\mathcal A})$ be the number of bounded cells having an odd number of edges.
We have:
$$I \times \delta(\mathcal{A}) = \sum_{i=1}^{I}\delta(P_i)
=\sum_{i=1}^{I}\left\lfloor \frac{f_1(P_i)}{2} \right\rfloor
= \sum_{i=1}^{I}\frac{f_1(P_i)}{2}-\frac{p_{odd}({\mathcal A})}{2}
=\frac{2f_1({\mathcal A})-f_1^0({\mathcal A})-p_{odd}({\mathcal A})}{2}.
$$
Since $f_1({\mathcal A})=n(n-2)$, to maximize $\delta(\mathcal{A})$ is equivalent
to minimize $f_1^0({\mathcal A})+p_{odd}({\mathcal A})$.
We clearly have $f_1^0({\mathcal A}^{o}_{\:2,n})=2(n-1)$, and this is the best possible as
the number of external edges $f_1^0({\mathcal A})$  is at least $2(n-1)$,
see~\cite{BDX07}.
We have have $p_{odd}({\mathcal A}^{o}_{\:2,n})=n-2$ for even $n$, and this is the best
possible since at least $n-2$ bounded cells of a simple line arrangement are triangles.
If $p_{odd}({\mathcal A})$ is odd, $\sum_{i=1}^I f_1(P_i)$ is odd.
If $f_1^0({\mathcal A}^{o}_{\:2,n})=2(n-1)$,  $\sum_{i=1}^I f_1(P_i)=2f_1({\mathcal A})-f_1^0({\mathcal A})$
is even. Thus, for odd $n$, $f_1^0({\mathcal A})+p_{odd}({\mathcal A})$ is at least
$2(n-1)+(n-2)+1$ which is achieved by ${\mathcal A}^{o}_{\:2,n}$.
Thus ${\mathcal A}^{o}_{\:2,n}$ minimizes $f_1^0({\mathcal A})+p_{odd}({\mathcal A})$;
that is, maximizes $\delta(\mathcal{A})$.
\end{proof}

\section{Plane Arrangements with Large Average Diameter} \label{3d}
For $n\geq 5$, we consider the simple plane arrangement
$\mathcal{A}^o_{\:3,n}$ made of the the $3$ planes $h_1$, $h_2$ and
$h_3$ corresponding, respectively, to $x_3=0$, $x_2=0$ and $x_1=0$,
and $(n-3)$ planes defined by their intersections with the $x_1$,
$x_2$ and $x_3$ axis. We have $h_k\cap h_1\cap
h_2=\{1+2(k-4)\varepsilon,0,0\}$, $h_k\cap h_1\cap
h_3=\{0,1+(k-4)\varepsilon,0\}$ and $h_k\cap h_2\cap
h_3=\{0,0,1-(k-4)\varepsilon\}$ for $k=4,5,\dots,n-1$, and $h_n\cap
h_1\cap h_2=\{3,0,0\}$, $h_n\cap h_1\cap h_3=\{0,2,0\}$ and $h_n\cap
h_2\cap h_3=\{0,0,3+\varepsilon\}$ where $\varepsilon$ is a constant
satisfying $0<\varepsilon<1/(n-4)$. See Figure~\ref{A073} for
an illustration of an arrangement combinatorially equivalent to
$\mathcal{A}^o_{\:3,7}$ where, for clarity, only the bounded cells
belonging to the positive orthant are drawn.

\begin{figure}[htb]
\begin{center}
\epsfig{file=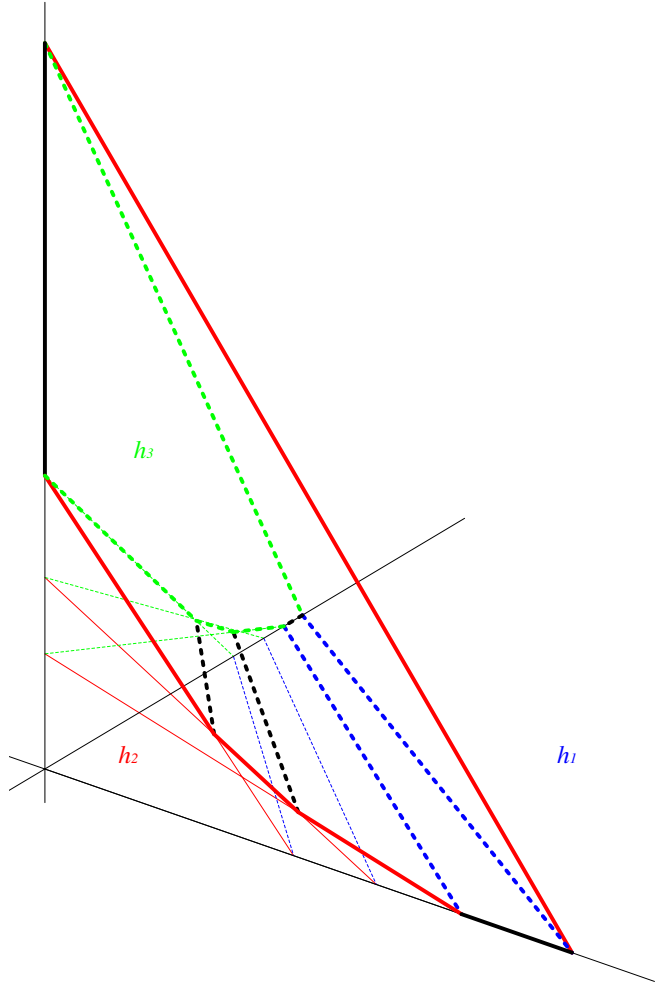, height= 13cm} \caption{An arrangement
combinatorially equivalent to $\mathcal{A}^o_{\:3,7}$} \label{A073}
\end{center}
\end{figure}

\begin{figure}[thb]
\begin{center}
\epsfig{file=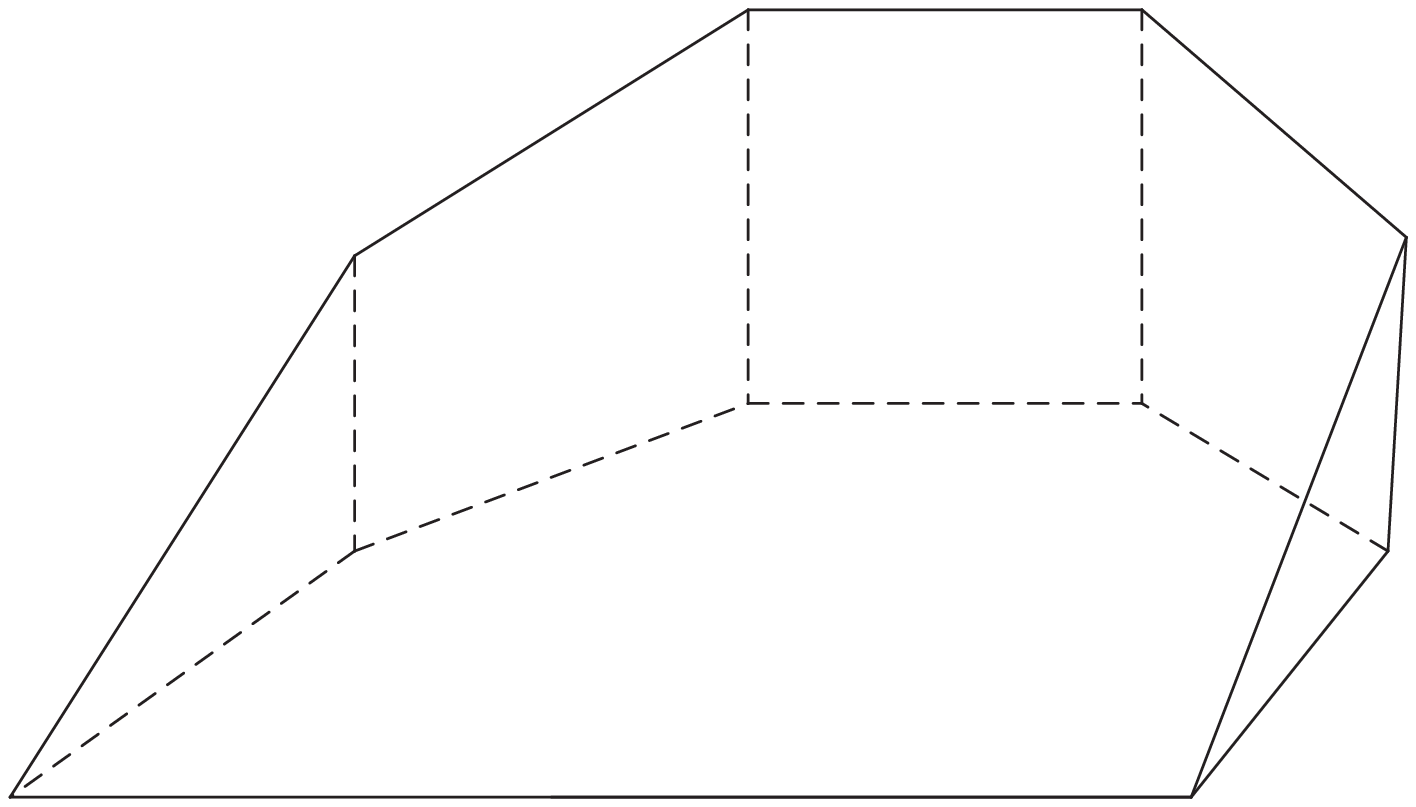, height=2.2cm} \caption{A polytope combinatorially equivalent to the shell $S_7$} \label{S72}
\end{center}
\end{figure}

\begin{prop}
For $n\geq 5$, the bounded cells of the  arrangement ${\mathcal
A}^{o}_{\:3,n}$ consist of $(n-3)$ tetrahedra, $(n-3)(n-4)-1$ cells
combinatorially equivalent to a prism with a triangular base, ${n-3
\choose 3}$ cells combinatorially equivalent to a cube, and $1$ cell
combinatorially equivalent to a shell $S_n$ with $n$ facets and
$2(n-2)$ vertices. See Figure~\ref{S72} for an illustration of
$S_7$. We have $\delta({\mathcal
A}^{o}_{\:3,n})=3-\frac{6}{n-1}+\frac{6(\lfloor\frac{n}{2}\rfloor-2)}{(n-1)(n-2)(n-3)}$
for $n\geq 5$.
\end{prop}
\begin{proof}
For $4\leq k\leq n-1$, let ${\mathcal A}^*_{\:3,k}$ denote the
arrangement formed by the first $k$ planes of ${\mathcal
A}^{o}_{\:3,n}$. See Figure~\ref{A63} for an arrangement
combinatorially equivalent to ${\mathcal A}^*_{3,6}$. We first show
by induction that the bounded cells of the arrangement ${\mathcal
A}^*_{\:3,n-1}$ consist of $(n-4)$ tetrahedra, $(n-4)(n-5)$
combinatorial triangular prisms and ${n-4 \choose 3}$ combinatorial
cubes. We use the following notation to describe the bounded cells
of ${\mathcal A}^*_{\:3,k-1}$: $T_{\triangle}$ 
for a tetrahedron with a facet on $h_1$
; $P_{\triangle}$, respectively $P_{\diamond}$, for a
combinatorial triangular prism with a triangular, respectively
square, facet on $h_1$;
$C_{\diamond}$ for a combinatorial cube with a square facet on $h_1$;
and  $C$, respectively $T$ and $P$, for a combinatorial cube,
respectively tetrahedron and triangular prism, not touching $h_1$.
When the plane $h_k$ is added, the cells $T_{\triangle}$,
$P_{\triangle}$, $P_{\diamond}$, and $C_{\diamond}$ are sliced,
respectively, into $T$ and $P_{\triangle}$, $P$ and
$P_{\triangle}$, $P$ and $C_{\diamond}$, and $C$ and $C_{\diamond}$. In
addition, one $T_{\triangle}$ cell and $(k -4)$ $P_{\diamond}$
cells are created by bounding 
$(k-3)$ unbounded cells of ${\mathcal A}^*_{\:3,k-1}$. Let $c(k)$
denotes the number of $C$ cells of ${\mathcal A}^*_{\:3,k}$,
similarly for $C_{\diamond}$, $T$, $T_{\triangle}$, $P$,
$P_{\triangle}$ and $P_{\diamond}$. For ${\mathcal
A}^*_{\:3,4}$ we have $t_{\triangle}(4)=1$ and $t(4)=p(4)=
p_{\triangle}(4)=p_{\diamond}(4)=c(4)=c(4)=0$. The
addition of $h_k$ removes and adds one $T_{\triangle}$, thus,
$t_{\triangle}(k) =1$. Similarly, all $P_{\diamond}$ are removed and
$(k-4)$ are added, thus, $p_{\diamond}(k)=(k-4)$. Since
$t(k)=t(k-1)+t_{\triangle}(k-1)$ and $p_{\triangle}(k)=
p_{\triangle}(k-1)+t_{\triangle}(k-1)$, we have $t(k)=
p_{\triangle}(k)=(k-4)$. Since $p(k)=p(k-1)+
p_{\triangle}(k-1)+p_{\diamond}(k-1)$, we have
$p(k)=(k-4)(k-5)$. Since $c_{\diamond}(k)=
c_{\diamond}(k-1)+p_{\diamond}(k-1)$, we have $c_{\diamond}(k)={k-4 \choose
2}$. Since $c(k)=c(k-1)+c_{\diamond}(k-1)$, we have $c(k)={k-4 \choose
3}$. Therefore the bounded cells of ${\mathcal A}^*_{\:3,n-1}$
consist of $t(n-1)+t_{\triangle}(n-1)=(n-4)$ tetrahedra, $p(n-1)+
p_{\triangle}(n-1)+p_{\diamond}(n-1)=(n-4)(n-5)$
combinatorial triangular prisms, and $c(n-1)+c_{\diamond}(n-1)={n-4
\choose 3}$ combinatorial cubes. The addition of $h_n$ to ${\mathcal
A}^*_{\:3,n-1}$ creates $1$ shell $S_n$ with $2$ triangular facets
belonging to $h_2$ and $h_3$ and $1$ square facet belonging to
$h_1$.
Besides $S_n$, all the bounded cells created by the addition of
$h_n$ are below $h_1$.
One $P_{\diamond}$ and $n-5$ combinatorial cubes are created between
$h_2$ and $h_3$. The other bounded cells are on the negative side of $h_3$:
$n-5$ $P_{\diamond}$ and $1$ $T_{\triangle}$ between $h_n$ and $h_{n-1}$,
and
$n-k-5$ $C_{\diamond}$ and $1$ $P_{\triangle}$ between $h_{n-k}$
and $h_{n-k-1}$ for $k=1,\dots,n-5$.
In total, we have $1$ tetrahedron, ${n-4 \choose 2}$ combinatorial cubes
and $(2n-9)$ combinatorial triangular prisms below $h_1$.
Since the diameter of a tetrahedron, triangular prism, cube and
$n$-shell is, respectively, $1,2,3$ and $\lfloor\frac{n}{2}\rfloor$,
we have $\delta({\mathcal
A}^{o}_{\:3,n})=3-6\frac{2(n-3)+(n-3)(n-4)-1-(\lfloor\frac{n}{2}\rfloor-3)}{(n-1)(n-2)(n-3)}
=3-\frac{6}{n-1}+\frac{6(\lfloor\frac{n}{2}\rfloor-2)}{(n-1)(n-2)(n-3)}$.
\end{proof}

\begin{figure}[htb]
\begin{center}
\epsfig{file=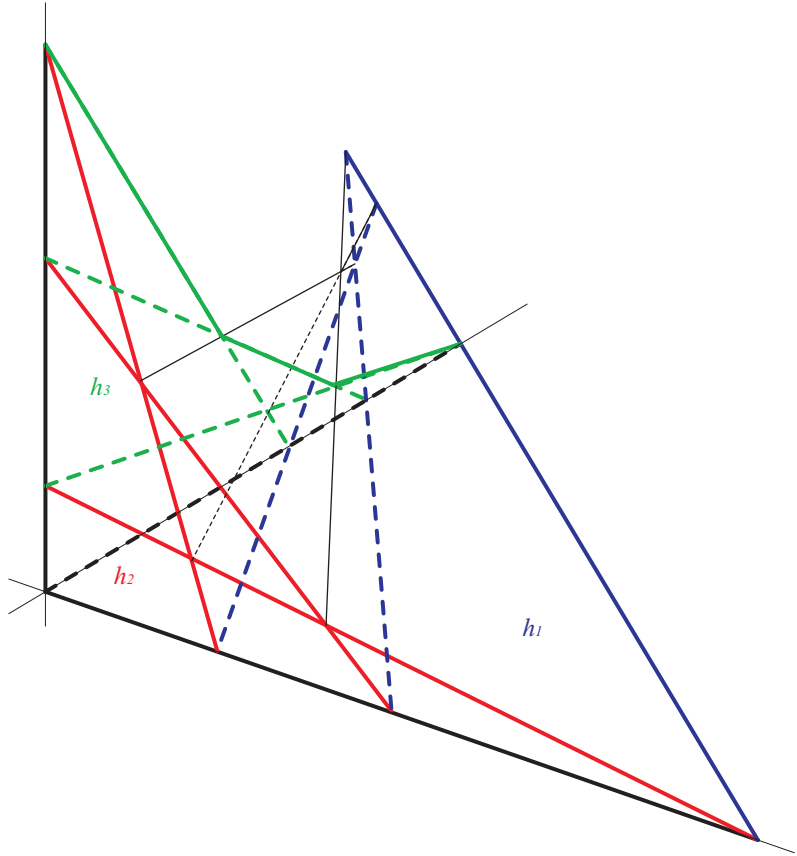, height= 11.3cm} \caption{An arrangement
combinatorially equivalent to $\mathcal{A}^*_{\:3,6}$} \label{A63}
\end{center}
\end{figure}

\begin{remark} There is only one combinatorial type of simple arrangement of $5$ planes,
and  we have $\Delta_{\mathcal{A}}(3,5)=\delta({\mathcal
A}^{o}_{\:3,5})=\frac{3}{2}$. Among the $43$ simple combinatorial types of
arrangements formed by $6$ planes~\cite{om}, the maximum average diameter is $2$
while  $\delta({\mathcal A}^{o}_{\:3,6})=1.8$. See Figure~\ref{A63_29}
for an illustration of the combinatorial type of one of the two
simple arrangements with $6$ planes maximizing the average
diameter. The far away vertex on the right and $3$ bounded edges incident to it
are cut off  (same for the far away vertex on the left) so the $10$ bounded cells
of the arrangement ($3$ tetrahedra, $4$ simplex prisms, and $3$ 6-shells)
 appear not too small.
\end{remark}

\begin{figure}[htb]
\begin{center}
\epsfig{file=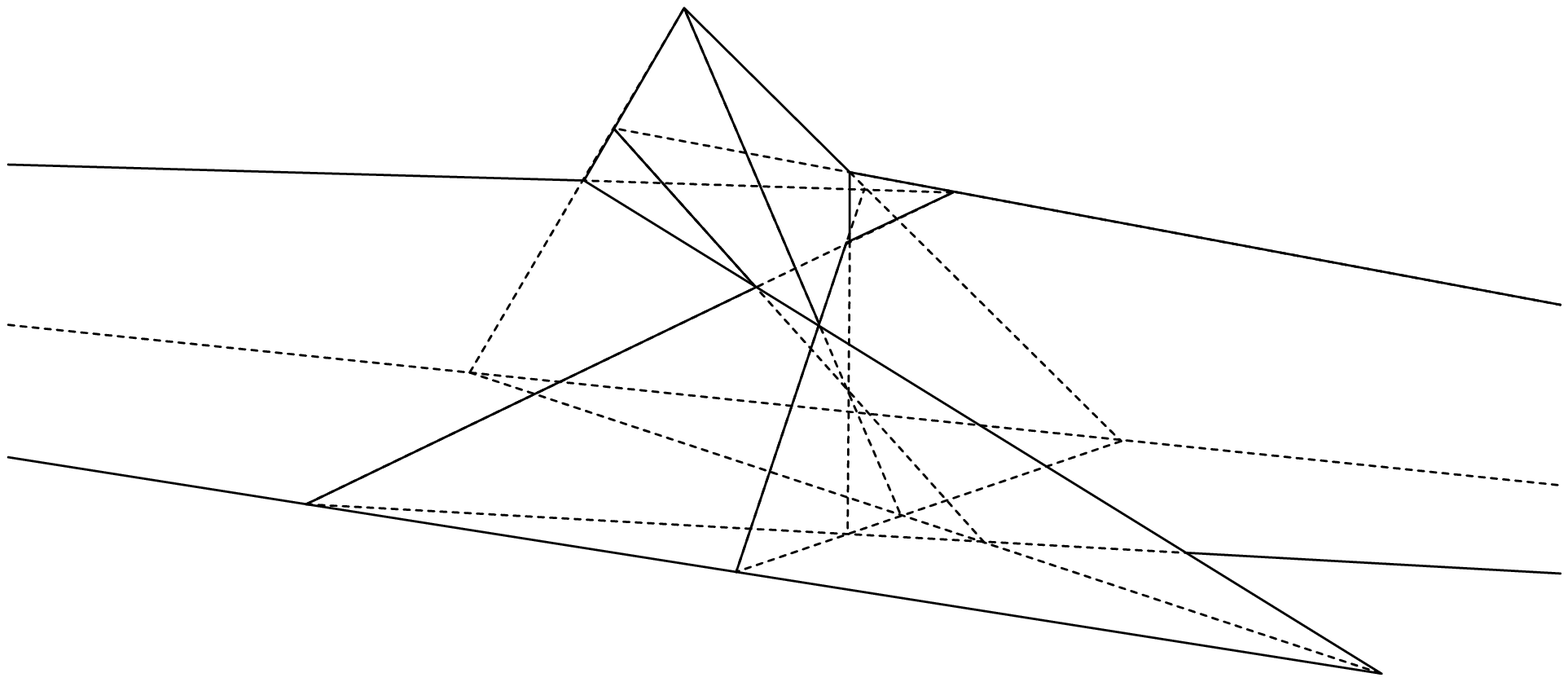, height= 6.3cm} \caption{An arrangement
formed by 6 planes maximizing the average diameter} \label{A63_29}
\end{center}
\end{figure}

\begin{prop}
For $n\geq 4$, the largest possible average diameter of a bounded cell of a simple
arrangement of $n$ planes satisfies
$3-\frac{6}{n-1}+\frac{6(\lfloor\frac{n}{2}\rfloor-2)}{(n-1)(n-2)(n-3)}\leq
\Delta_{{\mathcal A}}(3,n)\leq 3 + \frac{4(2n^2-16n+21)}{3(n-1)(n-2)(n-3)}$.
\end{prop}
\begin{proof}
Let $f_2({\mathcal A})$ denote the number of bounded facets of a simple arrangement ${\mathcal A}$
of $n$ planes, and let $f_2(P_i)$ denote the number of facets of a bounded cell $P_i$ of ${\mathcal A}$.
Let call a facet of ${\mathcal A}$ {\em external}  if it belongs to exactly one bounded cell,
and let $f_2^0({\mathcal A})$ denote the number of external facets of ${\mathcal A}$.
We have: $I\times \delta(\mathcal{A})=$
$$
\sum_{i=1}^{I} \delta(P_i)
 \leq\sum_{i=1}^{I} \left( \left\lfloor \frac{2f_2(P_i)}{3}\right\rfloor  - 1 \right)
\leq \sum_{i=1}^{I}  \frac{2f_2(P_i)}{3} -\frac{n-3}{3}-I
= \frac{4f_2({\mathcal A})-2f_2^0({\mathcal A})-n+3-3I}{3}$$
where the second inequality holds since at least $(n-3)$ bounded cells of ${\mathcal A}$
are simplices~\cite{S79}. Since $f_2({\mathcal A})=n{n-2 \choose 2}$ and $ f_2^0({\mathcal A})$ is at least
$\frac{n(n-2)}{3}+2$,
see~\cite{BDX07}, we have $\delta({\mathcal A})\leq 3 + 4(2n^2-16n+21)/3(n-1)(n-2)(n-3)$.
\end{proof}

\section{Hyperplane Arrangements with Large Average Diameter} \label{dd}
After recalling in Section~\ref{d+2} the unique
combinatorial structure of a simple arrangement formed by $d+2$
hyperplanes in dimension $d$, we show in Section~\ref{sdd}
 that the cyclic hyperplane
arrangement ${\mathcal A}^*_{\:d,n}$ contains ${n-d \choose
d}$ cubical cells for $n\geq 2d$. It implies that the average
diameter $\delta({\mathcal A}^*_{\:d,n})$ is arbitrarily close to
$d$ for $n$ large enough. Thus, the dimension $d$ is an asymptotic
lower bound for $\Delta_{\cal A}(d,n)$ for fixed $d$.

\begin{figure}[hbt]
\begin{center}
\epsfig{file=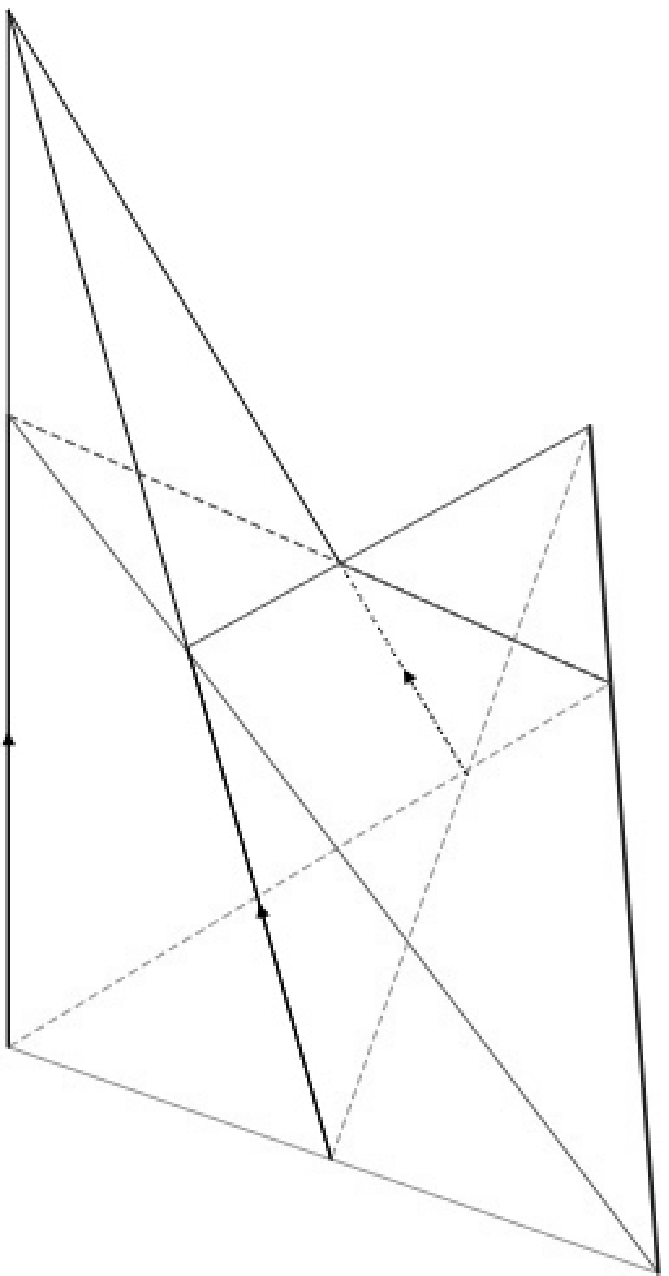, height=6.2cm} \caption{An arrangement
combinatorially equivalent to $\mathcal{A}_{\:3,5}$} \label{A53DR}
\end{center}
\end{figure}

\subsection{The average diameter of a simple arrangement with $d+2$ hyperplanes} \label{d+2}
Let ${\mathcal A}_{\:d,d+2}$ be a simple arrangement formed by $d+2$
hyperplanes in dimension $d$. Besides simplices, the bounded cells
of ${\mathcal A}_{\:d,d+2}$ are simple polytopes with $d+2$ facets
corresponding to the product of a $k$-simplex with a $(d-k)$-simplex for
$k=1,\dots,\lfloor\frac{d}{2}\rfloor$, see for example~\cite{G03}.
We recall one way to show that the combinatorial type of the arrangement of
$d+2$ hyperplanes in dimension $d$ is unique.
The affine Gale dual, see~\cite[Chapter 6]{Z93}, of the
$d+3$ vectors in dimension $d+1$ corresponding to the
linear arrangement associated to ${\mathcal A}_{\:d,d+2}$
(and the hyperplane at infinity)
forms a configuration of $d+3$ distinct signed points on a line; i.e.,
is unique up to relabeling and reorientation.
We also recall the combinatorial structure of
${\mathcal A}_{\:d,d+2}$ as some of the notions presented are used
in Section~\ref{sdd}. Since there is only one combinatorial type of
simple arrangement with $d+2$ hyperplanes, the arrangement
${\mathcal A}_{\:d,d+2}$ can be obtained from the simplex ${\mathcal
A}_{\:d,d+1}$ by cutting off one its vertices $v$ with the
hyperplane $h_{d+2}$. As a result, a prism $P$ with a simplex base is created.
Let us call {\em top base} the base of $P$ which belongs to
$h_{d+2}$ and assume, without loss of generality, that the
hyperplane containing the bottom base of $P$ is $h_{d+1}$. Besides
the simplex defined by $v$ and the vertices of the top base of $P$,
the remaining $d$ bounded cells of ${\mathcal A}_{\:d,d+2}$ are
between $h_{d+2}$ and $h_{d+1}$. See Figure~\ref{A53DR} for an
illustration the combinatorial structure of ${\mathcal A}_{\:3,5}$.
As the projection of ${\mathcal A}_{\:d,d+2}$ on $h_{d+1}$ is
combinatorially equivalent to ${\mathcal A}_{\:d-1,d+1}$, the $d$
bounded cells between $h_{d+2}$ and $h_{d+1}$ can be obtained from
the $d$ bounded cells of ${\mathcal A}_{\:d-1,d+1}$ by the {\em
shell-lifting} of ${\mathcal A}_{\:d-1,d+1}$ over the ridge
$h_{d+1}\cap h_{d+2}$; that is, besides the vertices belonging to
$h_{d+1}\cap h_{d+2}$, all the vertices in $h_{d+1}$ (forming
${\mathcal A}_{\:d-1,d+1}$) are lifted. See Figure~\ref{ADR} where
the skeletons of the $d+1$ bounded cells of ${\mathcal A}_{\:d,d+2}$
are given for $d=2,3,\dots,6$, and the shell-lifting of the bounded
cells is indicated by an arrow. The vertices not belonging to
$h_{d+1}$ are represented in black in Figure~\ref{ADR}, e.g., the
simplex cell containing $v$ is the one made of black vertices. The
bounded cells of ${\mathcal A}_{\:d,d+2}$ are $2$ simplices and a
pair of product of a $k$-simplex with a $(d-k)$-simplex for
$k=1,\dots,\lfloor\frac{d}{2}\rfloor$ for odd $d$. For even $d$ the product
of the $\frac{d}{2}$-simplex with itself is present only once. Since all the
bounded cells, besides the 2 simplices, have diameter $2$, we have
$\delta({\mathcal A}_{\:d,d+2})=\frac{2+2(d-1)}{d+1}$.
\begin{prop}
We have
$\Delta_{\mathcal{A}}(d,d+2)=\delta({\mathcal
A}_{\:d,d+2})=\frac{2d}{d+1}$.
\end{prop}

\begin{figure}[hbt]
\begin{center}
\epsfig{file=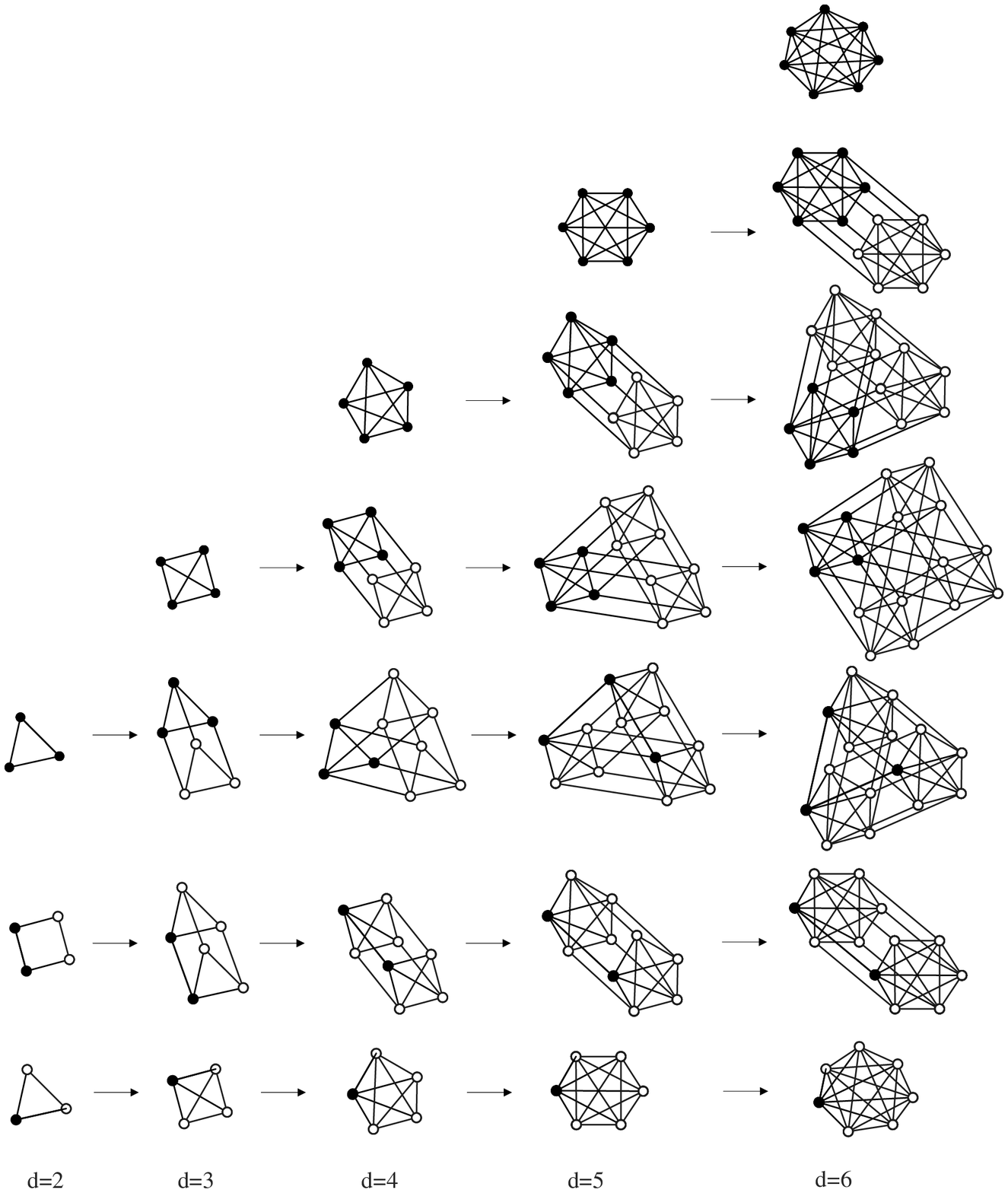, height=13.5cm} \caption{The skeletons of the
$d+1$ bounded cells of $\mathcal{A}_{\:d,d+2}$ for $d=2,3,\dots,6$.}
\label{ADR}
\end{center}
\end{figure}

\subsection{Hyperplane Arrangements with Large Average Diameter}\label{sdd}
We consider the simple hyperplane arrangement ${\mathcal A}^*_{\:d,n}$
combinatorially equivalent to the cyclic hyperplane and formed by the following $n$
hyperplanes $h^d_k$ for $k=1,2,\dots,n$. The hyperplanes
$h^d_k=\{x:x_{d+1-k}=0\}$ for $k=1,2,\dots,d$ form the positive
orthant, and the hyperplanes $h^d_k$ for $k=d+1,\dots,n$ are defined
by their intersections with the axes ${\bar x}_i$ of the positive
orthant. We have $h_k^d\cap {\bar
x}_i=\{0,\dots,0,1+(d-i)(k-d-1)\varepsilon,0,\dots,0\}$ for
$i=1,2\dots,d-1$ and $h_k^d\cap {\bar
x}_d=\{0,\dots,0,1-(k-d-1)\varepsilon\}$ where $\varepsilon$ is a
constant satisfying $0<\varepsilon<1/(n-d-1)$. The
combinatorial structure of $\mathcal{A}^*_{\:d,n}$ can be derived
inductively. All the bounded cells of $\mathcal{A}^*_{\:d,n}$ are on
the positive side of $h^d_1$ and $h^d_2$ with the bounded cells
between $h^d_2$ and $h^d_3$ being obtained by the shell-lifting of a
combinatorial equivalent of $\mathcal{A}^*_{\:d-1,n-1}$ over the
ridge $h^d_2 \cap h^d_3$, and the bounded cells on the other side of
$h^d_3$ forming a combinatorial equivalent of
$\mathcal{A}^*_{\:d,n-1}$. The intersection ${\mathcal A}^*_{\:d,n}
\cap h^d_k$  is combinatorially equivalent to ${\mathcal
A}^*_{\:d-1,n-1}$ for $k=2,3,\dots,d$ and removing $h^d_2$ from
${\mathcal A}^*_{\:d,n}$ yields an arrangement combinatorially
equivalent to ${\mathcal A}^*_{\:d,n-1}$. See Figure~\ref{A63} for
an arrangement combinatorially equivalent to
$\mathcal{A}^*_{\:3,6}$.
\begin{prop}\label{assymp}
The arrangement ${\mathcal A}^*_{\:d,n}$ contains $n-d \choose d$
cubical cells for $n\geq 2d$. We have $\delta({\mathcal A}^*_{\:d,n})\geq d{n-d \choose
d}/{n-1 \choose d}$ for $n\geq 2d$. It implies that for $d$ fixed,
$\Delta_{\cal A}(d,n)$ is arbitrarily close to $d$ for $n$ large
enough.
\end{prop}
\begin{proof}
The arrangements $\mathcal{A}_{\:n, 2}^{*}$ and $\mathcal{A}_{\:n,
3}^{*}$ contain, respectively, $n-2 \choose 2$ and $n-3 \choose 3$
cubical cells. The arrangement $\mathcal{A}_{\:d,2d}^{*}$ has $1$
cubical cell. Since ${\mathcal A}^*_{\:d,n}$ is obtained inductively
from ${\mathcal A}^*_{\:d,n-1}$ by lifting ${\mathcal
A}^*_{\:d-1,n-1}$ over the ridge $h^d_2\cap h^d_3$, we count
separately the cubical cells between $h^d_2$ and $h^d_3$ and the
ones on the other side of $h^d_3$. The ridge $h^d_2\cap h^d_3$ is an
hyperplane of the arrangements ${\mathcal A}^*_{\:d,n}\cap h^d_2$
and ${\mathcal A}^*_{\:d,n}\cap h^d_3$ which are both
combinatorially equivalent to ${\mathcal A}^*_{\:d-1,n-1}$. Removing
$h^{d-1}_2$ from ${\mathcal A}^*_{\:d,n}\cap h^d_2$ yields an
arrangement combinatorially equivalent to ${\mathcal
A}^*_{\:d-1,n-2}$. It implies that $(n-2)-(d-1) \choose d-1$ cubical
cells of ${\mathcal A}^*_{\:d,n}\cap h^d_2$ are not incident to the
ridge $h^d_2\cap h^d_3$. The shell-lifting of these $n-d-1 \choose
d-1$ cubical cells (of dimension $d-1$) creates $n-d-1 \choose d-1$
cubical cells between $h^d_2$ and $h^d_3$. As removing $h^d_2$ from
${\mathcal A}^*_{\:d,n}$ yields an arrangement combinatorial
equivalent to ${\mathcal A}^*_{\:d,n-1}$, there are ${n-1-d \choose
d}$ cubical cells on the other side of $h^d_3$. Thus, ${\mathcal
A}^*_{\:d,n}$ contains ${n-d-1\choose d-1}+{n-d-1\choose
d}={n-d\choose d}$ cubical cells.
\end{proof}

Proposition~\ref{assymp} can be slightly strengthened to the following proposition.
\newpage 
\begin{prop}
Besides $n-d \choose d$ cubical cells, the arrangement ${\mathcal A}^*_{\:d,n}$
contains $(n-d)$ simplices and $(n-d)(n-d-1)$ bounded cells
combinatorially equivalent to a prism with a simplex base for $n\geq 2d$.
We have $\Delta_{\cal A}(d,n)\geq 1+\frac{(d-1){n-d \choose
d}+(n-d)(n-d-1)}{{n-1 \choose d}}$ for $n\geq 2d$.
\end{prop}
\begin{proof}
Similarly to Proposition~\ref{assymp}, we can inductively count $(n-d)$ simplices and
$(n-d)(n-d-1)$ bounded cells of ${\mathcal A}^*_{\:d,n}$
combinatorially equivalent to a prism with a simplex base. We have
$(n-1)-(d-1)$ simplices in ${\mathcal A}^*_{\:d,n}\cap h^d_2$ and,
since removing $h^{d-1}_2$ from ${\mathcal A}^*_{\:d,n}\cap h^d_2$
yields an arrangement combinatorially equivalent to ${\mathcal
A}^*_{\:d-1,n-2}$, only one of these $(n-d)$ simplices of ${\mathcal
A}^*_{\:d,n}\cap h^d_2$ is incident to the ridge $h^d_2\cap h^d_3$.
Thus, between $h^d_2$ and $h^d_3$, we have $1$ simplex incident to
the ridge $h^d_2\cap h^d_3$ and $(n-d-1)$ cells combinatorially
equivalent to a prism with a simplex base not incident to the ridge
$h^d_2\cap h^d_3$. In addition, $(n-d-1)$ cells combinatorially
equivalent to a prism with a simplex base are incident to the ridge
$h^d_2\cap h^d_3$ and between $h^d_2$ and $h^d_3$. These $(n-d-1)$
cells correspond to the truncations of the simplex
$\mathcal{A}_{\:d,d+1}^{*}$ by $h^d_k$ for $k=d+2,d+3,\dots,n$.
Thus, we have $2(n-d-1)$ cells combinatorially equivalent to a prism
with a simplex base between $h^d_2$ and $h^d_3$. Since the other side
of $h^d_3$ is combinatorially equivalent to $\mathcal{A}_{\:n-1,
d}^{*}$, it contains $(n-1-d)$ simplices and $(n-d-1)(n-d-2)$
bounded cells combinatorially equivalent to a prism with a simplex
base. Thus, $\mathcal{A}_{\:d,n}^{*}$ has
$(n-d-1)(n-d-2)+2(n-d-1)=(n-d)(n-d-1)$ cells combinatorially
equivalent to a prism with a simplex base and $(n-d)$ simplices. As
a prism with a simplex base has diameter $2$ and the diameter of a
bounded cell is at least $1$, we have
 $\delta({\mathcal A}^*_{\:d,n})\geq \frac{d{n-d \choose d}+2(n-d)(n-d-1)+{n-1
 \choose d}-{n-d \choose d}-(n-d)(n-d-1)}{{n-1 \choose d}}$ for $n\geq 2d$.
\end{proof}

\noindent {\bf Acknowledgments}$\quad$ The authors would like to
thank Komei Fukuda, Hiroki Nakayama, and Christophe Weibel for use of their
codes~\cite{cdd,nak,minksum} which helped to investigate small simple arrangements.
The authors are grateful to the anonymous referees for many helpful suggestions and
for pointing out relevant papers~\cite{B87,Z93}. Research supported by an NSERC Discovery grant,  a
MITACS grant and  the Canada Research Chair program.

\end{document}